\documentclass[12pt,fleqn]{amsart}

\usepackage{setspace}
\RequirePackage{amssymb}
\usepackage{fullpage}
\usepackage{amsbsy}
\usepackage{enumerate}
\usepackage{mathrsfs}
\usepackage{amsmath}
\usepackage{url}
\usepackage{stmaryrd}

\theoremstyle{plain}
\newtheorem{thm}{Theorem}[section]
\newtheorem*{thm*}{Theorem}
\newtheorem*{mainthm*}{Main Theorem}
\newtheorem*{mainlem*}{Main Lemma}
\newtheorem{lem}[thm]{Lemma} \newtheorem*{lem*}{Lemma}
\newtheorem{claim}[thm]{Claim} \newtheorem*{claim*}{Claim}
\newtheorem{cor}[thm]{Corollary} \newtheorem*{cor*}{Corollary}
 \newtheorem*{prop*}{Proposition}

\theoremstyle{definition}
\newtheorem{defn}[thm]{Definition} \newtheorem*{defn*}{Definition}

\theoremstyle{remark}
\newtheorem{rem}[thm]{Remark} \newtheorem*{rem*}{Remark}
 \newtheorem*{example*}{Example}
 \newtheorem*{conj*}{Conjecture}
 \newtheorem*{question*}{Question}

\newcommand{\Ord}{\mathrm{Ord}}

\DeclareMathOperator{\powerset}{\mathcal{P}}

\DeclareMathOperator{\crit}{crit}

\DeclareMathOperator{\ran}{range}

\DeclareMathOperator{\Proj}{\boldsymbol{\pi}}
\DeclareMathOperator{\BP}{BP}
\DeclareMathOperator{\WF}{WF}

\begin{document}

\author{Trevor M.\ Wilson}
\title{The large cardinal strength of Weak Vop\v{e}nka's Principle}

\address{Department of Mathematics\\Miami University\\Oxford, Ohio 45056\\USA}
\email{twilson@miamioh.edu} 
\urladdr{https://www.users.miamioh.edu/wilso240}

\begin{abstract}
 We show that Weak Vop\v{e}nka's Principle, which is the statement that the opposite category of ordinals cannot be fully embedded into the category of graphs, is equivalent to the large cardinal principle Ord is Woodin, which says that for every class $C$ there is a $C$-strong cardinal. Weak Vop\v{e}nka's Principle was already known to imply the existence of a proper class of measurable cardinals. We improve this lower bound to the optimal one by defining structures whose nontrivial homomorphisms can be used as extenders, thereby producing elementary embeddings witnessing $C$-strongness of some cardinal.
\end{abstract}

\maketitle

\section{Introduction}

We work in the second-order set theory GB $+$ AC, meaning G\"{o}del--Bernays set theory with the Axiom of Choice for sets. This theory allows us to deal with with arbitrary classes and Ord-sequences of structures. Because every model of ZFC together with its definable classes forms a model of GB $+$ AC, the results of this paper also hold in ZFC for definable classes as a special case.

A \emph{graph} is a structure $\langle G; E\rangle$ where $G$ is a set and $E$ is a binary relation on $G$. A \emph{homomorphism} of graphs $\langle G; E\rangle \to \langle G'; E'\rangle$ is a function $h : G \to G'$ such that for all $\langle x_1,x_2\rangle \in E$ we have $\langle h(x_1),h(x_2)\rangle \in E'$. Sometimes we write just $G$ for a graph instead of $\langle G; E\rangle$. \emph{Vop\v{e}nka's Principle} (VP) states that the category of graphs has no large discrete full subcategory, or in other words that for every proper class of graphs there is some non-identity homomorphism among the graphs in that class.

Ad\'{a}mek, Rosick\'{y}, and Trnkov\'{a} \cite[Lemma 1]{AdaRosTrnLimitClosed} showed that VP is equivalent to the statement that the category $\Ord$ does not fully embed into the category of graphs, and they defined \emph{Weak Vop\v{e}nka Principle} (WVP) as the dual statement that the opposite category $\text{Ord}^\text{op}$ does not fully embed into the category of graphs. More explicitly:

\begin{defn}[\cite{AdaRosTrnLimitClosed}]
 \emph{Weak Vop\v{e}nka's Principle} (WVP) says that no sequence of graphs $\langle G_\alpha : \alpha \in \Ord\rangle$ has both of the following properties: whenever $\alpha \le \alpha'$ there is a unique homomorphism $G_{\alpha'} \to G_{\alpha}$ and whenever $\alpha < \alpha'$ there is no homomorphism $G_\alpha \to G_{\alpha'}$.
\end{defn}

Ad\'{a}mek, Rosick\'{y}, and Trnkov\'{a} \cite[Lemma 2]{AdaRosTrnLimitClosed} observed that that WVP follows from VP (hence the name ``weak''.) They also noted that it implies the existence of a proper class of measurable cardinals. The main result of this article is a substantial improvement to this large-cardinal lower bound, leading to an equivalence.

A variant of WVP is obtained by removing uniqueness from the definition:

\begin{defn}[\cite{AdaRosInjectivity}]
 \emph{Semi-Weak Vop\v{e}nka's Principle} (SWVP) says that no sequence of graphs $\langle G_\alpha : \alpha \in \Ord\rangle$ has both of the following properties: whenever $\alpha \le \alpha'$ there is a homomorphism $G_{\alpha'} \to G_{\alpha}$ and whenever $\alpha < \alpha'$ there is no homomorphism $G_\alpha \to G_{\alpha'}$.
\end{defn}

We have VP $\implies$ SWVP $\implies$ WVP, where the first implication is due to Ad\'{a}mek and Rosick\'{y} \cite{AdaRosInjectivity} and the second implication is trivial. Wilson \cite{WilWeakVopenka} proved that SWVP and WVP are equivalent to each other and are both strictly weaker than VP, in fact strictly weaker than the existence of a supercompact cardinal. In this article, we will show that the large cardinal principle  ``Ord is Woodin'' (defined below) implies SWVP and follows from WVP. Since SWVP trivially implies WVP, this gives the precise large cardinal strength of both principles. It also gives another proof of their equivalence to each other.
 
Before stating the definition of this large cardinal principle, we briefly review some set-theoretic terminology.

Recall that a class $M$ is called \emph{transitive} if every element of $M$ is a subset of $M$. Important examples of transitive sets are $V_\alpha$, $\alpha\in \Ord$, defined recursively by $V_0 = \emptyset$, $V_{\alpha+1} = \powerset(V_\alpha)$, and $V_\lambda = \bigcup_{\alpha < \lambda} V_\alpha$ if $\lambda$ is a limit ordinal.  In other words, $V_\alpha$ is the set of all sets of rank less than $\alpha$.  The class of all sets, $V$, is equal to $\bigcup_{\alpha \in \Ord} V_\alpha$ by the axiom of foundation. The \emph{critical point} of an elementary embedding $j : \langle V; \in \rangle \to \langle M; \in \rangle$, denoted by $\crit(j)$, is the least ordinal $\kappa$ such that $j(\kappa) \ne \kappa$.  If $j$ is not the identity, then it has a critical point $\kappa$ equal to the least rank of any set that is moved by $j$.  Moreover, $\kappa$ is a cardinal and $j(\kappa) > \kappa$.

\begin{defn}\label{defn:weakly-beta-C-strong}
 Let $C$ be a class, let $\kappa$ be a cardinal, and let $\beta$ be an ordinal.  We say that $\kappa$ is \emph{$\beta$-$C$-strong} if there is a transitive class $M$ and an elementary embedding
 \[j : \langle V; \in \rangle \to \langle M; \in \rangle\]
 such that
 \begin{enumerate}
  \item\label{item:crit} $\crit(j) = \kappa$,
  \item\label{item:V-beta} $V_\beta \subset M$, and
  \item\label{item:coherence} $j(C \cap V_\beta) \cap V_\beta = C \cap V_\beta$.
 \end{enumerate}
\end{defn}

Because condition \eqref{item:crit} implies that $\kappa$ is a measurable cardinal, we could omit from the definition the explicit requirement that $\kappa$ is a cardinal. Note that if $\kappa$ is $\beta$-$C$-strong then it is $\beta'$-$C$-strong for every ordinal $\beta' < \beta$, as witnessed by the same elementary embedding $j$.

We will sometimes write condition \eqref{item:coherence} in the abbreviated form $j(C) \cap V_\beta = C \cap V_\beta$, where for every class $C$ the class $j(C)$ is defined as $\bigcup_{\alpha \in \Ord} j(C\cap V_\alpha)$.  This produces an equivalent definition because $j(C)$ and $j(C\cap V_\beta)$ have the same intersection with $V_\beta$.

Many authors include a fourth condition $j(\kappa) > \beta$ in the definition of $\beta$-$C$-strong. This condition is not relevant to WVP, so it will not be used in this article. Adding this condition results in an equivalent definition of $\beta$-$C$-strong when $\beta$ is a successor ordinal,\footnote{See Kanamori \cite[Exercise 26.7(b)]{KanHigherInfinite} for the case $C = \emptyset$, which is based on the argument of Kunen \cite{KunElementary} for $\beta$-supercompactness and easily generalizes to an arbitrary class $C$.} so our omission of it does not affect the global notion of $C$-strongness that we define next.
 
 \begin{defn}
  Let $C$ be a class and let $\kappa$ be a cardinal.  We say that $\kappa$ is \emph{$C$-strong} if it is $\beta$-$C$-strong for every ordinal $\beta$.
 \end{defn}
 
 We may now define the large cardinal notion that we will show is equivalent to WVP. It is a standard (though lesser-used) variation of the definition of ``Woodin cardinal'':

\begin{defn}
 \emph{Ord is Woodin} means that for every class $C$ there is a $C$-strong cardinal.
\end{defn}

The main result of this article is the following.

\begin{thm}[GB + AC] \label{thm:main}
 The following statements are equivalent.
 \begin{enumerate}
  \item\label{item:Ord-is-Woodin} Ord is Woodin.
  \item\label{item:SWVP} SWVP.
  \item\label{item:WVP} WVP.
 \end{enumerate}
\end{thm}

The proof that Ord is Woodin implies SWVP is relatively easy, so we give it now.

\begin{proof}[Proof of \eqref{item:Ord-is-Woodin} $\implies$ \eqref{item:SWVP}]
 Assume that Ord is Woodin, let $G = \langle G(\alpha) : \alpha \in \Ord \rangle$ be a sequence of graphs, and assume that whenever $\alpha \le \alpha'$ there is a homomorphism $G(\alpha') \to G(\alpha)$. It will suffice to show that there is an ordinal $\kappa$ and a homomorphism $G(\kappa) \to G(\kappa+1)$.
 
 Because Ord is Woodin, there is a $G$-strong cardinal $\kappa$. Take $\beta$ sufficiently large that the ordered pair $\langle \kappa+1, G(\kappa+1)\rangle$, which is an element of the class function $G$, is also an element of $V_\beta$. Because $\kappa$ is $\beta$-$G$-strong, there is a transitive class $M$ and an elementary embedding 
 \[j : \langle V; \in \rangle \to \langle M; \in \rangle\]
 such that $\crit(j) = \kappa$, $V_\beta \subset M$, and $j(G) \cap V_\beta = G \cap V_\beta$. Then  $\langle \kappa+1, G(\kappa+1)\rangle$ is also an element of the class function $j(G)$ (which is a sequence of graphs in $M$,) meaning
 \[ j(G)(\kappa+1) = G(\kappa+1).\]
 The elementarity of $j$ implies that the restriction $j \restriction G(\kappa)$ is a homomorphism $G(\kappa) \to j(G(\kappa))$, and also implies that $j(G(\kappa)) = j(G)(j(\kappa))$, so what we have is a homomorphism
 \[j \restriction G(\kappa) : G(\kappa) \to j(G)(j(\kappa)).\]
 By our assumption on the existence of ``backward'' homomorphisms from later graphs to earlier ones, the elementarity of $j$, and the fact that $j(\kappa) > \kappa+1$, there is a homomorphism $h$ in $M$ from $j(G)(j(\kappa))$ to $j(G)(\kappa+1)$.  The fact that $h$ is a homomorphism is absolute between $M$ and $V$, and we have $j(G)(\kappa+1) = G(\kappa+1)$ as mentioned previously, so what we have is a homomorphism
 \[h : j(G)(j(\kappa)) \to G(\kappa+1).\]
 Then the composition $h \circ (j \restriction G(\kappa))$ is a homomorphism $G(\kappa) \to G(\kappa+1)$, as desired.
\end{proof}

To complete the proof of Theorem \ref{thm:main}, because SWVP trivially implies WVP it remains to prove that WVP implies Ord is Woodin. This will be done in Section \ref{section:WVP-implies-Woodin} (except the proof of Lemma \ref{lem:every-extender-is-derived}, which will be done in Section \ref{section:every-extender-is-derived}.)  In this proof, we will not work directly with graphs. We will use the fact that for any signature with countably many finitary relation (or partial operation) symbols, the category of structures in this signature fully embeds into the category of graphs by Hedrl\'{\i}n and Pultr \cite{HedPulFullEmbeddings}. The principles VP, WVP, and SWVP can therefore be taken to apply to such general structures instead of graphs.

The key concept in the proof that WVP implies Ord is Woodin will be a certain notion of extender. Various definitions of ``extender'' can be found in the literature. They fall into two main types, which are essentially equivalent. One type of extender, which we will not use in this article, is a family of ultrafilters whose ultrapower embeddings form a directed system: see for example Kanamori \cite[Section 26]{KanHigherInfinite}. Instead we will use the other type of extender, which is a function on power sets that preserves certain structure (and can therefore be thought of as a kind of homomorphism): see for example Neeman \cite[Definition 2.1]{NeeMitchellOrder}.
 
To keep this article simple and self-contained, we will define the most convenient structures on power sets whose homomorphisms suffice for our purposes. We will not rely on any pre-existing definitions or theorems about extenders.

Once Theorem \ref{thm:main} is proved, it can be applied locally to give a characterization of Woodin cardinals. Recall that a \emph{Woodin cardinal} is an inaccessible cardinal $\delta$ such that for every set $C \subset V_\delta$ there is a cardinal $\kappa < \delta$ that is $\beta$-$C$-strong for all $\beta < \delta$.

For an inaccessible cardinal $\delta$, the two-sorted structure $\langle V_\delta, V_{\delta+1}; \in \rangle$ satisfies GB + AC, where the elements of $V_\delta$ are regarded as sets of the structure and the elements of $V_{\delta+1}$ (subsets of $V_\delta$) are regarded as classes of the structure. It is not hard to see that an inaccessible cardinal $\delta$ is Woodin if and only if the structure $\langle V_\delta, V_{\delta+1}; \in \rangle$ satisfies ``Ord is Woodin.'' Applying Theorem \ref{thm:main} in this structure immediately yields the following consequence.

\begin{cor}[ZFC] \label{cor:Woodin}
 If $\delta$ is inaccessible, then the following statements are equivalent.
 \begin{enumerate}
  \item\label{item:delta-is-Woodin} $\delta$ is a Woodin cardinal.
   
  \item\label{item:SWVP-below-delta} There is no sequence of graphs $\langle G_\alpha : \alpha <\delta\rangle$ such that each graph $G_\alpha$ has cardinality less than $\delta$, whenever $\alpha \le \alpha' < \delta$ there is a homomorphism $G_{\alpha'} \to G_{\alpha}$, and whenever $\alpha < \alpha' < \delta$ there is no homomorphism $G_\alpha \to G_{\alpha'}$.
 
  \item\label{item:WVP-below-delta} There is no sequence of graphs $\langle G_\alpha : \alpha <\delta\rangle$ such that each graph $G_\alpha$ has cardinality less than $\delta$, whenever $\alpha \le \alpha' < \delta$ there is a unique homomorphism $G_{\alpha'} \to G_{\alpha}$, and whenever $\alpha < \alpha' < \delta$ there is no homomorphism $G_\alpha \to G_{\alpha'}$.
 \end{enumerate}
\end{cor}

\section{WVP implies Ord is Woodin}\label{section:WVP-implies-Woodin}

Assume GB + AC. For a set $X$ we write $X^{\mathord{<}\omega}$ for the set of all finite-length sequences of elements of $X$, and for $k<\omega$ we write $X^k$ for the set of all $k$-length sequences of elements of $X$.  We use the symbol $\powerset$ for the power set operation.

\begin{defn}
 For a set $X$ and natural numbers $j$, $k$, and $i_1,\ldots,i_j$ such that $1 \le i_1,\ldots,i_j \le k$, we define the function $\Proj_{k, \langle i_1,\ldots, i_j \rangle} :  X^k \to X^j$ by
 \begin{align*}
  & \Proj_{k, \langle i_1,\ldots, i_j \rangle}(\langle x_1, \ldots, x_k \rangle) = \langle x_{i_1}, \ldots, x_{i_j} \rangle.
 \end{align*}
 (The notation $\Proj$ is intended to indicate a vector of coordinate projections.)
\end{defn}

The following structures are designed in such a way that homomorphisms between them will correspond to elementary embeddings with domain $\langle V ; \in\rangle$.

\begin{defn}
 A \emph{$\mathscr{P}$-structure} is a structure
 \[\mathscr{P}_X = \big\langle \powerset(X^{\mathord{<}\omega}); \mathord{\cap}, \mathord{-}, X^k, \WF, \Proj^{-1}_{k, \langle i_1,\ldots,i_j \rangle}, \BP_k \big\rangle_{j,k<\omega \text{ and } 1 \le i_i,\ldots,i_j \le k,}\]
 where $X$ is a transitive set, with the following operations and relations.
 \begin{enumerate}
  \item $\cap$ is the binary operation of intersection.
  \item $-$ is the unary operation of complementation.
  \item $X^k$ is a constant.
  \item $\WF$ is the unary relation on $\powerset(X^{\mathord{<}\omega})$ consisting of all sets $A \subset X^{\mathord{<}\omega}$ that are wellfounded under $\supsetneq$, the reverse of the proper initial segment relation.
 \item $\Proj^{-1}_{k, \langle i_1,\ldots,i_j \rangle}$ is the function $\powerset(X^j) \to \powerset(X^k)$, considered as a partial unary operation on $\powerset(X^{\mathord{<}\omega})$, that is the inverse image function of $\Proj_{k, \langle i_1,\ldots,i_j \rangle}$:
 \[ \Proj^{-1}_{k, \langle i_1,\ldots, i_j \rangle}(A) = \big\{\langle x_1, \ldots, x_k\rangle \in X^k : \langle x_{i_1}, \ldots, x_{i_j} \rangle  \in A\big\}.\]
 \item $\BP_k$ (for \emph{bounded projection}) is the function $\powerset(X^{k+1}) \to \powerset(X^{k+1})$, considered as a partial unary operation on $\powerset(X^{\mathord{<}\omega})$, defined by
 \[\BP_k(A) = \big\{ \langle x_1, \ldots, x_{k+1} \rangle \in X^{k+1}: \exists z \in x_{k+1}\; \langle x_1,\ldots, x_{k},z \rangle \in A \big\}. \]
 \end{enumerate}
\end{defn}

Our main results will only use the structures $\mathscr{P}_{V_\beta}$ for ordinals $\beta$, but we may as well allow $X$ to be an arbitrary transitive set in the definition. Note that for every set $A \subset X^{\mathord{<}\omega}$, the following statements are equivalent by DC:
\begin{itemize}
 \item $\WF(A)$ fails.
 \item There is an infinite chain $a_1 \subsetneq a_2 \subsetneq a_3 \subsetneq \cdots$ of elements of $A$.
 \item There is a sequence $f \in X^\omega$ such that $f\restriction n \in A$ for infinitely many $n < \omega$.
\end{itemize}
The bounded projection operator $\BP_k$ is called ``bounded'' because of the bounded existential quantifier $\exists z \in x_{k+1}$ in the definition.
Note that these operators $\BP_k$ depend on the structure of $X$ as a material set, meaning $\langle X; \in \rangle$.

A \emph{homomorphism} of $\mathscr{P}$-structures is a homomorphism of structures in the usual sense: it is a function that preserves (commutes with) the operations and partial operations, and preserves the relation $\WF$, but does not necessarily preserve failure of $\WF$. Because all boolean operations are generated by $\cap$ and $-$, every homomorphism of $\mathscr{P}$-structures is a homomorphism of boolean algebras, and in particular preserves the subset relation $\subset$.

\begin{rem}
 In the usual terminology of extenders, the properties of preservation of $\WF$, $\Proj^{-1}$, and $\BP$ are called \emph{countable completeness}, \emph{coherence}, and \emph{normality} respectively.
\end{rem}

The homomorphisms given by the following lemma will be called \emph{trivial homomorphisms}. Although they carry no information, their existence will be crucial for our application because they will provide the ``backward'' homomorphisms in the definition of WVP.

\begin{lem}\label{lem:trivial-hom}
 For all transitive sets $X$ and $Y$ such that $Y \subset X$, the function $h : \powerset(X^{\mathord{<}\omega}) \to \powerset(Y^{\mathord{<}\omega})$ defined by $h(A) = A \cap Y^{\mathord{<}\omega}$ is a homomorphism from $\mathscr{P}_X$ to $\mathscr{P}_Y$.
\end{lem}
\begin{proof}
 Clearly $h$ is a boolean algebra homomorphism, $h(X^k) = Y^k$ for all $k < \omega$, and $h$ preserves the unary relation $\WF$. Preservation of the unary partial operation $\Proj^{-1}_{k, \langle i_1,\ldots,i_j \rangle}$ follows from closure of $Y^{\mathord{<}\omega}$ under the function $\Proj_{k, \langle i_1,\ldots,i_j\rangle}$. To verify preservation of the unary partial operation $\BP_k$ we must show that for all $A \subset X^{\mathord{<}\omega}$ and $y_1,\ldots, y_{k+1}\in Y$,
 \[   \langle y_1,\ldots, y_{k+1}\rangle \in \BP_k(A) \iff \langle y_1,\ldots, y_{k+1}\rangle \in \BP_k(A \cap Y^{\mathord{<}\omega}).\]
 In other words, we must show that
 \[  \exists z \in y_{k+1} \; \langle y_1,\ldots, y_k,z\rangle \in A \iff \exists z \in y_{k+1} \; \langle y_1,\ldots, y_k,z\rangle \in A \cap Y^{\mathord{<}\omega}.\]
 This follows from transitivity of $Y$: if $y_{k +1}\in Y$ and $z \in y_{k+1}$, then $z \in Y$.  (Here we rely on the ``boundedness'' of the bounded projection operator.)
\end{proof}

Nontrivial homomorphisms of $\mathscr{P}$-structures can be obtained from elementary embeddings of the set-theoretic universe $V$ by the usual ``derived extender'' method:

\begin{lem}\label{lem:derived-extender}
 Let $j : \langle V; \in \rangle \to \langle M; \in \rangle$ be an elementary embedding for some transitive class $M$. For all transitive sets $X$ and $Y$ such that $Y \subset j(X)$, the function $h : \powerset(X^{\mathord{<}\omega}) \to \powerset(Y^{\mathord{<}\omega})$ defined by $h(A) = j(A) \cap Y^{\mathord{<}\omega}$ is a homomorphism from $\mathscr{P}_X$ to $\mathscr{P}_Y$.
\end{lem}
\begin{proof}
 Because $j$ is elementary and the $\mathscr{P}$-structure $\mathscr{P}_X$ is uniformly definable from $X$, the restriction $j \restriction \mathcal{P}(X^{\mathord{<}\omega})$ is a homomorphism from $\mathscr{P}_X$ to $(\mathscr{P}_{j(X)})^M$.  Here $(\mathscr{P}_{j(X)})^M$ means the structure $\mathscr{P}_{j(X)}$ as it is defined in $M$.  Note that $j(X)$ is a transitive set, $(j(X)^{\mathord{<}\omega})^M = j(X)^{\mathord{<}\omega}$, and $\powerset(j(X)^{\mathord{<}\omega})^M \subset \powerset(j(X)^{\mathord{<}\omega})$ where the power set is computed in $M$.
 
 It follows that $(\mathscr{P}_{j(X)})^M$ is a substructure of $\mathscr{P}_{j(X)}$ because the definitions of the operations and relations are easily seen to be absolute for transitive models of ZFC. (In particular, every instance of the $\WF$ relation that holds in $M$ holds also in $V$, because it can be ``certified'' by an ordinal-valued rank function in $M$ and the ordinals of $M$ are the true ordinals.)  By composing $j \restriction \mathcal{P}(X^{\mathord{<}\omega})$ with the inclusion homomorphism from $(\mathscr{P}_{j(X)})^M$ to $\mathscr{P}_{j(X)}$, we may therefore consider it as a homomorphism from $\mathscr{P}_X$ to $\mathscr{P}_{j(X)}$.
 
 The function $h$ is the composition of the homomorphism $j \restriction \mathcal{P}(X^{\mathord{<}\omega}) : \mathscr{P}_X \to \mathscr{P}_{j(X)}$ with the trivial homomorphism $\mathscr{P}_{j(X)} \to \mathscr{P}_Y$ that exists by Lemma \ref{lem:trivial-hom} because $Y \subset j(X)$, so it is a homomorphism.
\end{proof}

A homomorphism $h$ given by an elementary embedding $j : \langle V; \in \rangle \to \langle M; \in \rangle$ as in Lemma \ref{lem:derived-extender} is said to be \emph{derived from $j$}. We will need the following result, which states that every homomorphism of $\mathscr{P}$-structures can be realized as a derived homomorphism.  It will be proved in the next section, where we will build $M$ and $j$ from $h$ using a standard ``ultrapower'' (also called ``term model'') construction.

\begin{lem}\label{lem:every-extender-is-derived}
 Let $X$ and $Y$ be transitive sets and let $h : \mathscr{P}_X \to \mathscr{P}_Y$ be a homomorphism.  Then there is a transitive class $M$ and an elementary embedding $j : \langle V; \in \rangle \to \langle M; \in \rangle$ such that $Y \subset j(X)$ and
 $h(A) = j(A) \cap Y^{\mathord{<}\omega}$ for all $A \subset X^{\mathord{<}\omega}$.
\end{lem}

The above definitions and results about $\mathscr{P}$-structures would suffice to obtain a strong cardinal (meaning a $C$-strong cardinal for $C = \emptyset$) from WVP, but to obtain a $C$-strong cardinal for an arbitrary class $C$ we will need a slight addition to the notion of $\mathscr{P}$-structure:
 
\begin{defn}
 A \emph{pointed $\mathscr{P}$-structure} is a $\mathscr{P}$-structure with an additional constant:
 \[\mathscr{P}_{X,c} = \big\langle \powerset(X^{\mathord{<}\omega}); \mathord{\cap}, \mathord{-}, X^k, \WF, \Proj^{-1}_{k, \langle i_1,\ldots,i_j \rangle}, \BP_k, c^{\mathord{<}\omega}\big\rangle_{j,k<\omega \text{ and } 1 \le i_i,\ldots,i_j \le k}\]
 where $X$ is a transitive set and $c \subset X$.
\end{defn}

The notion of homomorphism for pointed $\mathscr{P}$-structures is defined in the usual way, so a homomorphism $\mathscr{P}_{X,c} \to \mathscr{P}_{Y,d}$ is a homomorphism  $h: \mathscr{P}_{X} \to \mathscr{P}_{Y}$ such that $h(c^{\mathord{<}\omega}) = d^{\mathord{<}\omega}$. In particular, for a class $C$, a homomorphism $\mathscr{P}_{X,C \cap X} \to \mathscr{P}_{Y, C\cap Y}$ is a homomorphism  $h: \mathscr{P}_{X} \to \mathscr{P}_{Y}$ such that $h((C\cap X)^{\mathord{<}\omega}) = (C \cap Y)^{\mathord{<}\omega}$.

We may now complete the proof of Theorem \ref{thm:main} by showing that if WVP holds, then Ord is Woodin. Assume that Ord is not Woodin, meaning that for some class $C$ there is no $C$-strong cardinal. Define the class of ordinals
\[ S = \{ \beta \in \Ord : \forall \kappa < \beta \,(\text{$\kappa$ is not $\beta$-$C$-strong})\}.\]

\begin{claim}
 $S$ is a proper class.
\end{claim}
\begin{proof}
 For every ordinal $\kappa$, we may define $f(\kappa)$ to be the least ordinal $\beta$ such that $\kappa$ is not $\beta$-$C$-strong, which exists because $\kappa$ is not $C$-strong.  (If $\kappa$ is not a measurable cardinal, then $\beta = 0$ works here.) Then $S$ contains the class of all closure points of the function $f : \Ord \to \Ord$, which is a closed unbounded class of ordinals.
\end{proof}

We may therefore enumerate $S$ in strictly increasing order as $S = \{ \beta(\xi) : \xi \in \Ord\}$ and define an Ord-sequence of pointed $\mathscr{P}$-structures $\big\langle \mathscr{P}_{V_{\beta(\xi)}, C \cap V_{\beta(\xi)}} : \xi \in \Ord\big\rangle.$ As mentioned in the introduction, WVP equivalently applies to more general structures such as pointed $\mathscr{P}$-structures instead of graphs, so it remains to prove the following claim.

\begin{claim}
 $\big\langle \mathscr{P}_{V_{\beta(\xi)}, C \cap V_{\beta(\xi)}} : \xi \in \Ord\big\rangle$ is a counterexample to WVP.
\end{claim}
\begin{proof}
 If not, then because we have trivial ``backward'' homomorphisms given by Lemma \ref{lem:trivial-hom}, there must also be some nontrivial homomorphism \[h : \mathscr{P}_{V_{\alpha}, C \cap V_{\alpha}} \to \mathscr{P}_{V_{\beta}, C \cap V_{\beta}}\] for some ordinals $\alpha, \beta \in S$. By ``nontrivial'' we mean that either $\alpha < \beta$, or else $\alpha \ge \beta$ and $h(A) \ne A \cap V_\beta^{\mathord{<}\omega}$ for some set $A \subset V_\alpha^{\mathord{<}\omega}$.

 Considering $h$ as a homomorphism from $\mathscr{P}_{V_{\alpha}}$ to $\mathscr{P}_{V_\beta}$, by Lemma \ref{lem:every-extender-is-derived} there is a transitive class $M$ and an elementary embedding $j : \langle V; \in \rangle \to \langle M; \in \rangle$ such that $h$ is derived from $j$, meaning $V_\beta \subset j(V_{\alpha})$ and $h(A) = j(A) \cap V_\beta^{\mathord{<}\omega}$ for all $A \subset V_{\alpha}^{\mathord{<}\omega}$. Moreover, because $h$ is a homomorphism of pointed $\mathscr{P}$-structures it preserves the additional constant for $C$, meaning $h((C \cap V_\alpha)^{\mathord{<}\omega}) = (C \cap V_\beta)^{\mathord{<}\omega}$. It follows that $j(C \cap V_\alpha) \cap V_\beta = C \cap V_\beta$.
 
 We will obtain a contradiction to $\beta \in S$ by showing that $j$ witnesses the definition of $\beta$-$C$-strongness for some $\kappa < \beta$.
 
 First, note that $j$ moves some element of $V_\beta$. If $\alpha < \beta$ then this follows from the fact that $V_\beta \subset j(V_{\alpha})$. If $\alpha \ge \beta$ then this follows from the nontriviality assumption that $h(A) \ne A \cap V_\beta^{\mathord{<}\omega}$ for some set $A \subset V_\alpha^{\mathord{<}\omega}$, which implies that $j(A) \cap V_\beta^{\mathord{<}\omega} \ne A \cap V_\beta^{\mathord{<}\omega}$, so $j$ moves some element of $V_\beta^{\mathord{<}\omega}$. Therefore $\crit(j)$ exists and is less than $\beta$, so we may define $\kappa = \crit(j)$. 
 
 Second, note that the condition $V_\beta \subset j(V_{\alpha})$ implies $V_\beta \subset M$ because $j(V_{\alpha}) \in M$ and $M$ is transitive. Third and finally, note that the $C$-preservation condition $j(C \cap V_\alpha) \cap V_\beta = C \cap V_\beta$ implies $j(C \cap V_\beta) \cap V_\beta = C \cap V_\beta$ because $j(\alpha) \ge \beta$.  Therefore $\kappa$ is $\beta$-$C$-strong, which is a contradiction.
\end{proof}

\section{Proof of Lemma \ref{lem:every-extender-is-derived}}\label{section:every-extender-is-derived}

Let $X$ and $Y$ be transitive sets and let $h : \mathscr{P}_X \to \mathscr{P}_Y$ be a homomorphism. We want to show there is a transitive class $M$ and an elementary embedding $j : \langle V; \in \rangle \to \langle M; \in \rangle$ such that the homomorphism $\mathscr{P}_X \to \mathscr{P}_Y$ derived from $j$ is equal to $h$, meaning that $Y \subset j(X)$ and 
 $h(A) = j(A) \cap Y^{\mathord{<}\omega}$ for all $A \subset X^{\mathord{<}\omega}$.

Our structure $\langle M; \in\rangle$ will be obtained by a standard ``ultrapower'' (also called ``term model'') construction, similar to Neeman \cite{NeeMitchellOrder} or Zeman \cite{ZemInnerModels}. As a first approximation, we build a structure $\langle M^*; \in^*, =^*\rangle$, in which the symbols $\in$ and $=$ are interpreted as binary relations $\in^*$ and $=^*$ respectively, rather than true membership and equality.

\begin{defn} $M^* = \{ \langle k, b, f \rangle : k < \omega \text{ and } b \in Y^k \text{ and } f : X^k \to V\}$.

The binary relations $\in^*$ and $=^*$ on $M^*$ are defined by
  \begin{align*}
  \langle k_1, b_1, f_1 \rangle \in^* \langle k_2, b_2, f_2 \rangle &\iff b_1 b_2 \in h\big(\big\{a_1 a_2 : f_1(a_1) \in f_2(a_2)\big\}\big),\\
  \langle k_1, b_1, f_1 \rangle =^* \langle k_2, b_2, f_2 \rangle &\iff b_1 b_2 \in h\big(\big\{a_1 a_2 : f_1(a_1) = f_2(a_2)\big\}\big),
 \end{align*}
where juxtaposition (as in $a_1a_2$ and $b_1b_2$) denotes concatenation of finite sequences. 
\end{defn}

Note that in definitions of sets like $\{a_1 a_2 : f_1(a_1) \in f_2(a_2)\}$, we implicitly assume the condition $a_i \in X^{k_i}$ that is required to make sense of the expression $f_i(a_i)$.

Next we prove a version of \L o\'{s}'s theorem for the structure $\langle M^*; \in^*, =^*\rangle$.  Here ``formula'' means a formula in the first order language with two binary relation symbols $\in$ and $=$.
 
\begin{claim}\label{claim:Los}
 For every formula $\varphi$ and every $n<\omega$ sufficiently large that all free variables of $\varphi$ are contained in the set $\{v_1,\ldots,v_n\}$, the following statements are equivalent for all elements $\langle k_1,b_1,f_1 \rangle,\ldots,\langle k_n,b_n,f_n\rangle$ of $M^*$:
 \begin{enumerate}
  \item $\langle M^*; \in^*, =^* \rangle \models \varphi\big[ \langle k_1,b_1,f_1\rangle, \ldots, \langle k_n,b_n,f_n \rangle\big]$.
  \item $b_1 \cdots b_n \in h\big(\big\{ a_1 \cdots a_n :  \varphi[f_1(a_1),\ldots,f_n(a_n)]\big\}\big)$.\footnote{Here by $\varphi[f_1(a_1),\ldots,f_n(a_n)]$ we mean $\langle V; \in, =\rangle \models \varphi[f_1(a_1),\ldots,f_n(a_n)]$.}
 \end{enumerate}
 Moreover, there is an elementary embedding
 \[j^* : \langle V; \in , =\rangle \to \langle M^*; \in^*, =^*\rangle\] defined by
 $ j^*(p) = \langle 0, \Diamond, c_p \rangle $
 where $\Diamond$ is the empty sequence and for every $p \in V$ the constant function $c_p : \{\Diamond\} \to V$ is defined by $c_p(\Diamond) = p$.  
\end{claim}
\begin{proof}
 The first part is proved by induction on formulas. For the first base case, assume that $\varphi$ is the formula $v_i \in v_j$ and let $n \ge \max\{i,j\}$. Then by definition of $\in^*$ we have
 \begin{align*}
  &\langle M^* ; \in^*, =^* \rangle \models \varphi\big[\langle k_1,b_1,f_1\rangle,\ldots,\langle k_n,b_n,f_n\rangle\big]\\
  \iff
  &\langle k_i, b_i, f_i\rangle \in^* \langle k_j, b_j, f_j \rangle\\
  \iff
  &b_ib_j \in h\big(\big\{ a_ia_j : f_i(a_i) \in f_j(a_j)\big\}\big)\\
  \iff
  &b_1\cdots b_n \in h\big(\big\{ a_1\cdots a_n : f_i(a_i) \in f_j(a_j)\big\}\big) && (\Proj^{-1})\footnotemark\\
  \iff
  &b_1 \cdots b_n \in h\big(\big\{ a_1 \cdots a_n : \varphi[f_1(a_1),\ldots,f_n(a_n)]\big\}\big).
 \end{align*}
 \footnotetext{This notation means that we use the fact that $h$ preserves some $\Proj^{-1}$ operator in this step. More specifically, the operator $\Proj^{-1}_{k_1 + \cdots + k_n,
  \langle
   k_1 + \cdots + k_{i-1}+1, \ldots, k_1 + \cdots + k_i,
   k_1 + \cdots + k_{j-1}+1, \ldots, k_1 + \cdots + k_j
  \rangle}$.}The second base case (where $\varphi$ is the formula $v_i = v_j$) is similar, using the definition of $=^*$ instead of the definition of $\in^*$.
  
 Now assume that $\varphi$ is $\neg \psi$ where the claim holds for $\psi$.  Then letting $k = k_1 + \cdots + k_n$,
 \begin{align*}
  &\langle M^*; \in^*, =^* \rangle \models \varphi\big[ \langle k_1,b_1,f_1\rangle, \ldots, \langle k_n,b_n,f_n \rangle\big]\\
   \iff & \langle M^*; \in^*, =^* \rangle \not\models \psi\big[ \langle k_1,b_1,f_1\rangle, \ldots, \langle k_n,b_n,f_n \rangle\big]\\
   \iff & b_1 \cdots b_n \in Y^k \setminus h\big(\big\{ a_1 \cdots a_n :  \psi[f_1(a_1),\ldots,f_n(a_n)]\big\}\big)\\
   \iff & b_1 \cdots b_n \in h\big( X^k \setminus \big\{ a_1 \cdots a_n :  \psi[f_1(a_1),\ldots,f_n(a_n)]\big\}\big)\\
   \iff & b_1 \cdots b_n \in h\big(\big\{ a_1 \cdots a_n :  \varphi[f_1(a_1),\ldots,f_n(a_n)]\big\}\big),
 \end{align*}
 using preservation of boolean operations and the fact that $h(X^k) = Y^k$.

 If $\varphi$ is $\psi_1 \wedge \psi_2$ where the claim holds for $\psi_1$ and $\psi_2$, then
 the claim holds for $\varphi$ by an entirely straightforward argument using preservation of $\cap$.
 
 Finally, assume that $\varphi$ is obtained by existential quantification from a formula $\psi$ for which the claim holds. For simplicity of notation, we assume that $n = 2$ and $\varphi$ is $\exists v_3 \psi$.  (The general case can be proved similarly, with only notational complications.) Then we have
  \begin{align*}
  &\langle M^*; \in^*, =^*  \rangle \models \varphi\big[\langle k_1,b_1,f_1\rangle, \langle k_2,b_2,f_2\rangle \big]\\
  \implies
  &\langle M^*; \in^*, =^*  \rangle \models \psi\big[\langle k_1,b_1,f_1 \rangle, \langle k_2,b_2,f_2\rangle, \langle k_3,b_3,f_3\rangle \big] \text{ f.s.\ $\langle k_3,b_3,f_3\rangle \in M^*$}\\
  \implies
  & b_1 b_2 b_3 \in h\big(\big\{ a_1 a_2 a_3: \psi[f_1(a_1),f_2(a_2),f_3(a_3)]\big\}\big)\\
  \implies
  & b_1 b_2 b_3 \in h\big(\big\{ a_1 a_2 a_3: \varphi[f_1(a_1),f_2(a_2)]\big\}\big) && (\subset)\\
  \implies
  & b_1 b_2  \in h\big(\big\{ a_1 a_2: \varphi[f_1(a_1),f_2(a_2)]\big\}\big), && (\Proj^{-1})\footnotemark
 \end{align*}
 \footnotetext{More specifically, $\Proj^{-1}_{k_1 + k_2 + k_3, \langle 1,\ldots,k_1+k_2 \rangle}$.}and conversely, letting $k_3 = k_1 + k_2$ and $b_3 = b_1 b_2$, we have
 \begin{align*}
  &b_1 b_2 \in h\big(\big\{ a_1 a_2 : \varphi[f_1(a_1),f_2(a_2)]\big\}\big)\\
  \implies 
  &b_1 b_2 \in h\big(\big\{ a_1 a_2 : \psi[f_1(a_1),f_2(a_2),f_3(a_1a_2)]\big\}\big) \text{ f.s.\ $f_3 : X^{k_3} \to V$} && \text{(AC)}\\
  \implies 
  &b_1 b_2 b_1 b_2 \in h\big(\big\{ a_1 a_2 a_1 a_2: \psi[f_1(a_1),f_2(a_2),f_3(a_1a_2)]\big\}\big) && (\Proj^{-1})\footnotemark \\
  \implies 
  &b_1 b_2 b_3 \in h\big(\big\{ a_1 a_2 a_3: \psi[f_1(a_1),f_2(a_2),f_3(a_3)]\big\}\big) && (\subset)\\
  \implies
  &\langle M^*; \in^*, =^*  \rangle \models \psi\big[\langle k_1,b_1,f_1\rangle, \langle k_2,b_2,f_2\rangle, \langle k_3,b_3,f_3\rangle \big]\\
  \implies
  &\langle M^*; \in^*, =^*  \rangle \models \varphi\big[ \langle k_1,b_1,f_1\rangle, \langle k_2,b_2,f_2\rangle \big],
 \end{align*}
 where in the step labeled AC we use the Axiom of Choice to produce a function $f_3$ choosing witnesses for the existential quantifier whenever they exist.
 \footnotetext{More specifically, $\Proj^{-1}_{k_3, \langle 1,\ldots,k_3,1,\ldots,k_3\rangle}$.}
 
 For the ``moreover'' part, let $p_1,\ldots,p_n \in V$.  Then we have
 \begin{align*}
  &\langle M^*; \in^*, =^*\rangle \models \varphi[j^*(p_1),\ldots,j^*(p_n)]\\
   \iff
   & \langle M^*; \in^*, =^*\rangle \models \varphi\big[ \langle 0, \Diamond, c_{p_1}\rangle,\ldots, \langle 0, \Diamond, c_{p_n}\rangle\big]\\
   \iff
   & \Diamond \cdots \Diamond \in h\Big(\Big\{\Diamond \cdots \Diamond : \varphi \big[c_{p_1}(\Diamond),\ldots,c_{p_n}(\Diamond)\big]\Big\}\Big)\\
   \iff
   & \Diamond \in h\big(\big\{\Diamond : \varphi[p_1,\ldots,p_n]\big\}\big)\\
   \iff
   & \varphi[p_1,\ldots,p_n],
 \end{align*}
 because $h(\emptyset) = \emptyset$ and $h(\{\Diamond\}) = h(X^0) = Y^0 =  \{\Diamond\}$.
\end{proof}

The existence of the elementary embedding $j^*$ implies that the structure $\langle M^*; \in^*, =^* \rangle$ is elementarily equivalent to $\langle V; \in, = \rangle$. It follows that the relation $=^*$ is an equivalence relation and that the relation $\in^*$ is invariant under $=^*$. We may therefore define the quotient of the structure $\langle M^*; \in^*, =^* \rangle$ by the relation $=^*$ to obtain a structure that interprets the equality symbol as true equality:
 
\begin{defn}
 $\langle M'; \in', = \rangle$ is the quotient structure $\langle M^*; \in^*, =^* \rangle / =^*$.
\end{defn}

Then we have 
\[M'= \big\{ [k, b, f] : k < \omega \text{ and } b \in Y^k \text{ and } f: X^k \to V\big\},\]
where $[k,b,f]$ denotes the equivalence class\footnote{We only include representatives of minimal rank in order to ensure that $[k, b,f]$ is a set (Scott's trick).} of $\langle k, b, f \rangle$ under $=^*$ in $M^*$. A straightforward induction on formulas shows that the quotient function from $\langle M^*; \in^*,=^*\rangle$ to $\langle M'; \in', = \rangle$ given by $\langle k, b, f\rangle \mapsto [k, b, f]$ preserves the truth values of all formulas, so we may restate \L o\'{s}'s theorem (Claim \ref{claim:Los}) for the quotient structure $\langle M'; \in'\rangle$ as follows. (The equality symbol will henceforth always be interpreted as true equality. We will drop it from our notation for the structure, but it may still be used in formulas.)
 
\begin{claim}\label{claim:Los2}
 For every formula $\varphi$ and every $n<\omega$ sufficiently large that all free variables of $\varphi$ are contained in the set $\{v_1,\ldots,v_n\}$, the following statements are equivalent for all elements $[ k_1,b_1,f_1 ],\ldots,[ k_n,b_n,f_n]$ of $M'$:
 \begin{enumerate}
  \item $\langle M'; \in' \rangle \models \varphi\big[ [ k_1,b_1,f_1], \ldots, [ k_n,b_n,f_n ]\big]$.
  \item $b_1 \cdots b_n \in h\big(\big\{ a_1 \cdots a_n :  \varphi[f_1(a_1),\ldots,f_n(a_n)]\big\}\big)$.
 \end{enumerate}
 Moreover, there is an elementary embedding
 \[j' : \langle V; \in \rangle \to \langle M'; \in'\rangle\] defined by
 $ j'(p) = [ 0, \Diamond, c_p ] $
 where $\Diamond$ is the empty sequence and for every $p \in V$ the constant function $c_p : \{\Diamond\} \to V$ is defined by $c_p(\Diamond) = p$.  
\end{claim}
 
The next step is to replace the structure $\langle M'; \in'\rangle$ with an isomorphic structure $\langle M; \in\rangle$ where $M$ is a transitive set and $\in$ is the true membership relation. We will do this using the Mostowski collapse.
The existence of the elementary embedding $j'$ implies that the structure $\langle M'; \in'\rangle$ satisfies the Axiom of Extensionality, so to show that its Mostowski collapse exists, it remains to verify the following two claims.
 
\begin{claim}\label{claim:well-founded}
 The relation $\in'$ on $M'$ is well-founded.
\end{claim}
\begin{proof}
 Here we use the fact that our homomorphism $h : \mathscr{P}_X \to \mathscr{P}_Y$ preserves the relation $\WF$. Suppose toward a contradiction that $\in'$ is not well-founded.  Then by DC there is an infinite decreasing sequence
 \[ [ k_1,b_1,f_1 ] \ni' [ k_2,b_2,f_2 ] \ni' [ k_3,b_3,f_3 ] \ni' \cdots,\]
 so for all $n < \omega$ we have
 \[ b_1 \cdots b_n \in h(\{a_1\cdots a_n : f_{n}(a_{n}) \in \cdots \in f_1(a_1)\})\]
 by \L o\'{s}'s theorem (Claim \ref{claim:Los2}). Because $h$ preserves $\subset$ it follows that for all $n<\omega$,
 \[ b_1 \cdots b_n \in h(A) \text{ where } A = \bigcup_{n <\omega} \{a_1\cdots a_n : f_{n}(a_{n}) \in \cdots \in f_1(a_1)\}.\]
 Therefore $h(A)$ contains the infinite chain $\{b_1 \cdots b_n : n <\omega\}$, and because $h$ preserves $\WF$ it follows that $A$ also contains some infinite chain. However, an infinite chain in $A$ would produce an infinite decreasing $\in$-sequence in $V$, contradicting the well-foundedness of $\in$.
\end{proof}
 
\begin{claim}\label{claim:set-like}
 The relation $\in'$ on $M'$ is set-like.
\end{claim}
\begin{proof}
 Let $[k_2, b_2, f_2]  \in M'$.
 We will show there is only a set (rather than a proper class) of elements $[k_1, b_1, f_1]  \in M'$
 such that $[k_1,b_1, f_1] \in' [k_2,b_2, f_2]$. Given $[k_1,b_1, f_1] \in' [k_2,b_2, f_2]$, we will define a ``small'' approximation
 $\bar{f}_1$ to $f_1$ that is bounded by $f_2$ in a certain sense.  Namely, we define the function
 $\bar{f}_1 : X^{k_1} \to V$ by
 \[\bar{f}_1(a_1) = \begin{cases}
             f_1(a_1) & \text{if $f_1(a_1) \in \bigcup\ran(f_2)$,}\\
             \emptyset & \text{otherwise.}
            \end{cases}\]
 (There is nothing special about $\emptyset$ here; we could use any other fixed value.)
 There is only a set of possibilities for $\bar{f}_1$ because of the restriction on its range, so it suffices to show that
 $[k_1,b_1,f_1]$ is equal to $[k_1,b_1, \bar{f}_1]$.  Indeed, we have
 \begin{align*}
  [k_1,b_1,f_1] \in' [k_2,b_2,f_2]
  \implies
  & b_1b_2 \in h\big(\big\{a_1a_2 : f_1(a_1) \in f_2(a_2)\big\}\big)\\
  \implies
  & b_1b_2 \in h\big(\big\{a_1a_2 : f_1(a_1) = \bar{f}_1(a_1)\big\}\big) && (\subset)\\
  \implies
  & b_1b_1 \in h\big(\big\{a_1a_1 : f_1(a_1) = \bar{f}_1(a_1)\big\}\big) && (\Proj^{-1})\footnotemark\\
  \implies
  & b_1b_1 \in h\big(\big\{a_1\bar{a}_1 : f_1(a_1) = \bar{f}_1(\bar{a}_1)\big\}\big) && (\subset)\\
  \implies
  &[k_1,b_1,f_1] = [k_1,b_1, \bar{f}_1],
 \end{align*}
 where $\bar{a}_1$ denotes an arbitrary element of $X^{k_1}$ not necessarily equal to $a_1$.
\end{proof}
 
\footnotetext{More specifically, $\Proj^{-1}_{k_1+k_2, \langle 1, \ldots, k_1,1,\ldots, k_1\rangle}$.}
  
We may therefore define the Mostowski collapse:
  
\begin{defn}
 $\langle M; \in\rangle$ is the Mostowski collapse of $\langle M'; \in' \rangle$.
\end{defn}

Then we have
\[M = \big\{ \llbracket k, b, f\rrbracket : k < \omega \text{ and } b \in Y^k \text{ and } f: X^k \to V\big\},\]
where $\llbracket k,b,f\rrbracket$ denotes the image of $[k,b,f]$ under the Mostowski collapse function. Because the Mostowski collapse function is an isomorphism, we may restate \L o\'{s}'s theorem (Claim \ref{claim:Los2}) for the collapsed structure $\langle M; \in \rangle$ as follows.
 
\begin{claim}\label{claim:Los3}
 For every formula $\varphi$ and every $n<\omega$ sufficiently large that all free variables of $\varphi$ are contained in the set $\{v_1,\ldots,v_n\}$, the following statements are equivalent for all elements $\llbracket k_1,b_1,f_1 \rrbracket,\ldots,\llbracket k_n,b_n,f_n\rrbracket$ of $M$:
 \begin{enumerate}
  \item $\langle M; \in\rangle \models \varphi\big[ \llbracket k_1,b_1,f_1\rrbracket, \ldots, \llbracket k_n,b_n,f_n \rrbracket\big]$.
  \item $b_1 \cdots b_n \in h\big(\big\{ a_1 \cdots a_n :  \varphi[f_1(a_1),\ldots,f_n(a_n)]\big\}\big)$.
 \end{enumerate}
  Moreover, there is an elementary embedding
 \[j : \langle V; \in \rangle \to \langle M; \in\rangle\] defined by
 $ j(p) = \llbracket 0, \Diamond, c_p \rrbracket$
 where $\Diamond$ is the empty sequence and for every $p \in V$ the constant function $c_p : \{\Diamond\} \to V$ is defined by $c_p(\Diamond) = p$.  
\end{claim}

It remains to prove that the homomorphism derived from this elementary embedding $j$ is equal to $h$. We will need the following claim. (Note that the claim immediately implies $Y \subset M$, which is part of our desired conclusion):
 
\begin{claim}
 For all $y \in Y$, we have $y = \llbracket 1, \langle y \rangle, \pi\rrbracket$ where $\pi : X^1 \to X$ is the trivial projection function defined by $\pi(\langle x \rangle) = x$.
\end{claim}
\begin{proof}
 Here we use the fact that our homomorphism $h$ preserves the bounded projection operators $\BP_k$. The proof is by $\in$-induction on $y \in Y$.\footnote{For a similar argument pertaining to the other type of extender, see Martin and Steel \cite[Lemma 1.5]{MarSteProjectiveDeterminacy}.}  Let $y \in Y$ and assume that $y' = \llbracket 1, \langle y' \rangle, \pi\rrbracket$ for all $y' \in y$. This implies that $y \subset M$.  Also, we have $\llbracket 1, \langle y \rangle, \pi\rrbracket \subset M$ by transitivity of $M$. For every element $\llbracket k,b,f\rrbracket$ of $M$, we have
 \begin{align*}
  \llbracket k,b,f\rrbracket \in \llbracket 1, \langle y \rangle, \pi\rrbracket  \iff
   & b \langle y \rangle \in h\big(\big\{ a \langle x \rangle : f(a) \in \pi(\langle x\rangle)\big\}\big)\\
   \iff
   & b \langle y \rangle \in h\big(\big\{ a \langle x \rangle : f(a) \in x\big\}\big)\\
   \iff
   & b \langle y \rangle \in h\big(\BP_k\big(\big\{ a \langle x \rangle : f(a) = x\big\}\big)\big)\\
   \iff
   & b \langle y \rangle \in \BP_k\big(h\big(\big\{ a \langle x \rangle : f(a) = x\big\}\big)\big)\\
   \iff
   & (\exists y' \in y)\;  b \langle y' \rangle \in h\big(\big\{ a \langle x \rangle : f(a) = x\big\}\big)\\
   \iff
   & (\exists y' \in y)\;  b \langle y' \rangle \in h\big(\big\{ a \langle x \rangle : f(a) = \pi(\langle x \rangle)\big\}\big)\\
   \iff
   & (\exists y' \in y)\;  \llbracket k,b,f\rrbracket = \llbracket 1, \langle y' \rangle, \pi\rrbracket\\
   \iff
   & (\exists y' \in y)\;  \llbracket k,b,f\rrbracket = y'\\
   \iff
   & \llbracket k,b,f\rrbracket \in y,
 \end{align*}
 so the sets $\llbracket 1, \langle y \rangle, \pi\rrbracket$ and $y$ are subsets of each other and are therefore equal.
\end{proof}

Finally we will show that $h$ is derived from $j$:
 
\begin{claim}
 For all $A \subset X^{\mathord{<}\omega}$, we have $h(A) = j(A) \cap Y^{\mathord{<}\omega}$.
\end{claim}
\begin{proof}
 Let $A \subset X^{\mathord{<}\omega}$.  Then $h(A) \subset Y^{\mathord{<}\omega}$ and for all $\langle y_1,\ldots,y_k \rangle \in Y^{\mathord{<}\omega}$ we have
 \begin{align*}
  &\langle y_1,\ldots,y_k \rangle \in j(A)\\
  \iff
  &\langle M; \in \rangle \models \langle y_1,\ldots,y_k \rangle \in j(A)\\
  \iff
  &\langle M; \in \rangle \models \big\langle \llbracket 1, \langle y_1 \rangle, \pi\rrbracket,\ldots,
  \llbracket 1, \langle y_k \rangle, \pi\rrbracket \big\rangle \in \llbracket 0, \Diamond, c_A\rrbracket\\
  \iff
  & \langle y_1 \rangle \cdots \langle y_k \rangle \Diamond \in h\Big(\Big\{ 
  \langle x_1 \rangle \cdots \langle x_k \rangle \Diamond: \big\langle \pi(\langle x_1\rangle) ,\ldots \pi(\langle x_k\rangle) \big\rangle \in 
  c_A(\Diamond) \Big\} \Big)\\
  \iff
  &\langle y_1, \ldots y_k \rangle \in h\big(\big\{\langle x_1,\ldots, x_k \rangle : \langle x_1,\ldots, x_k \rangle \in 
  A \big\} \big) \\
  \iff
  & \langle y_1, \ldots y_k \rangle \in h(A). && \qedhere
 \end{align*}
\end{proof}
 
Applying this claim to the set $A = X^{\mathord{<}\omega}$ yields $h(X^{\mathord{<}\omega}) = j(X^{\mathord{<}\omega}) \cap Y^{\mathord{<}\omega}$. On the other hand, because $h$ is a boolean homomorphism we have $h(X^{\mathord{<}\omega}) = Y^{\mathord{<}\omega}$. It follows that $Y^{\mathord{<}\omega} \subset j(X^{\mathord{<}\omega})$ and therefore $Y \subset j(X)$, completing the proof of Lemma \ref{lem:every-extender-is-derived}.

\section{Acknowledgments}

The author thanks Joan Bagaria for several corrections to an earlier version of this paper, and Martin Zeman and Takehiko Gappo for helpful conversations about extenders.

\bibliographystyle{plain}
\bibliography{WVP}

\end{document}